\documentclass[12pt]{amsart}
\usepackage{caption}
\usepackage{url} 
\usepackage{hyperref}
\usepackage{amsfonts,amsthm,latexsym,amsmath,amssymb,amscd,amsmath,epsf} 
\usepackage{tikz}

\theoremstyle{plain}
\newtheorem{theorem}{\bf Theorem}[section]

\newtheorem{proposition}[theorem]{\bf Proposition}
\newtheorem{corollary}[theorem]{\bf Corollary}

\newtheorem{problem}[theorem]{\bf Open Problem}

\theoremstyle{definition}

\newtheorem{example}[theorem]{Example}

\theoremstyle{remark}
\newtheorem{remark}[theorem]{Remark}

\numberwithin{equation}{section}

\makeindex

\newcommand{\xx}{\mathbf{x}}

\newcommand{\QQ}{\mathbb{Q}}
\newcommand{\NN}{\mathbb{N}}
\newcommand{\PP}{\mathbb{P}}
\newcommand{\CC}{\mathbb{C}}
\newcommand{\ZZ}{\mathbb{Z}}
\newcommand{\Q}{\mathcal{Q}}
\newcommand{\sym}{\mathfrak{S}}
\newcommand{\T}{\mathcal{T}}

\newcommand{\ndrd}{\textsf{NDRD}}
\newcommand{\ndnl}{\textsf{NDNL}}
\newcommand{\naas}{\textsf{NAAS}}
\newcommand{\ntns}{\textsf{NTNS}}
\newcommand{\clr}{\mathbf{color}}
\newcommand{\nor}{\textsf{Nor}}
\newcommand{\lie}{\mathcal{L}ie}
\newcommand{\comb}{\textsf{Comb}}
\newcommand{\lyn}{\textsf{Lyn}}
\newcommand{\mcomb}{\textsf{MComb}}

\def\newop#1{\expandafter\def\csname #1\endcsname{\mathop{\rm #1}\nolimits}}

\newop{rdes}
\newop{p}
\newop{nlyn}
\newop{aapair}
\newop{tnpair}
\newop{red}
\newop{des}
\newop{free}
\newop{wcomp}

\title[]{A note on the $\gamma$-coefficients of the ``tree Eulerian polynomial"}
\author[R. S. Gonz\'alez D'Le\'on]{Rafael S. Gonz\'alez D'Le\'on$^*$}
\address{Department of Mathematics, University of Kentucky, Lexington, KY 40506-0027}
\email{rafaeldleon@uky.edu}
\thanks{$^*$This work was supported by NSF Grant  DMS 1202755}

\begin{document}
\begin{abstract}
We consider the generating polynomial of the number of rooted trees on the set 
$\{1,2,\dots,n\}$ counted by the number of descending edges (a parent with a greater label than a 
child). This polynomial is an extension of the 
descent generating polynomial of the set of permutations of a totally ordered $n$-set, known as 
the Eulerian polynomial. We show how this extension shares some of the properties of the 
classical one. B. Drake proved that this polynomial factors completely over the integers.
From his product formula it can be concluded that this polynomial has positive coefficients in 
the $\gamma$-basis and we show that a formula for these coefficients can also be derived. 
We discuss various combinatorial interpretations of these positive 
coefficients in terms of leaf-labeled binary trees and in terms of the 
Stirling permutations introduced by Gessel and Stanley. These interpretations are derived from 
previous results of the author and Wachs related to the poset of weighted 
partitions and the free multibracketed Lie algebra.
\end{abstract}

\maketitle
\section{introduction}\label{section:introduction}

A \emph{labeled rooted tree} $T$ on the set $[n]:=\{1,2,\cdots,n\}$ is a tree whose nodes or 
vertices are the elements of $[n]$ and such that one of its nodes has been distinguished and called 
the \emph{root}. For nodes $x$ and $y$ in $T$ we say that \emph{$x$ is the child of $y$} or 
\emph{$y$ is the parent of $x$} if $y$ is the first node following $x$ in the unique path from $x$ 
to 
the root of $T$ and we say that $y=\p(x)$. Nodes that have children are said to be \emph{internal} 
otherwise we call a node
without children a \emph{leaf}. If 
$y$ is the parent of $x$, we say that the edge $\{x,y\}$ of $T$ is \emph{descending} (and we call 
$x$ a \emph{descent} of $T$) if the label of $y$ is greater than the label of $x$. We denote 
$\des(T)$ the number of descents in $T$. Figure 
\ref{figure:rooted_trees_n3} shows all the rooted trees on $[3]$ grouped by the number of 
descents. We 
draw the trees with the convention that parents come higher than their children and the root is the 
highest node. We denote $\T_n$ the set of rooted trees on $[n]$ and 
$\T_{n,i}$ the set of trees in $\T_n$ with exactly $i$ descents. 

For a given $n \ge 1$ define

\begin{align}\label{equation:generating_polynomial}
 T_n(t):=\sum_{T \in \T_n} t^{\des(T)}=\sum_{i = 0}^{n-1} 
|\T_{n,i}| t^i,
\end{align}

the descent generating polynomial of $\T_n$. We call $T_n(t)$ the \emph{tree Eulerian 
polynomial}
in analogy with the classical polynomial $A_n(t)=\sum_{\sigma \in 
\sym_n}t^{\des(\sigma)}$, that is the descent generating polynomial of the set $\sym_n$ 
of permutations of $[n]$. We can identify the permutations in $\sym_n$ with the set of rooted 
trees on $[n]$ that have $n-1$ internal nodes, each of them having a unique child (and so 
containing a unique leaf). It is not hard to see that our definition of a descent on this set 
of trees coincides with the classical definition of descent in a permutation so the 
polynomial $T_n(t)$ is an extension of the polynomial $A_n(t)$. The polynomials 
$A_n(t)$ have been extensively studied in the literature and are known with the name of 
\emph{Eulerian polynomials} since Euler was one of the first in studying them (see 
\cite{Stanley2012}). The Eulerian polynomial  $A_n(t)=\sum_{k=0}^{n-1} A_{n,i}t^{i}$ have degree 
$n-1$ and its coefficients satisfy the relation

\begin{align}\label{equation:symmetry}
A_{n,i}=A_{n,n-1-i}. 
\end{align}

For example, the Eulerian polynomial for $n=3$ is $A_3(t)=1+4t+t^2$. A polynomial that 
satisfies Equation \ref{equation:symmetry} is called \emph{symmetric} or \emph{palindromic}.
It is a simple observation that a symmetric polynomial 
$f(t)=\sum_{k=0}^{d} f_{i}t^{i}$ of degree $d$ with $f_i \in \ZZ$ can be written in the form
\begin{align}
 \sum_{i = 0}^{d} 
f_{i} t^i=\sum_{i=0}^{\lfloor \frac{d}{2} \rfloor} \gamma _i t^i(1+t)^{d-2i},
\end{align}

where the coefficients $\gamma _i \in \ZZ$, i.e., the set $\{t^i(1+t)^{d-2i}\}_{i=0}^{\lfloor 
\frac{d}{2}\rfloor}$, where $\lfloor \cdot \rfloor$ is the integer floor function, is a basis 
(known as the \emph{$\gamma$-basis}) for the space of symmetric 
polynomials of degree $d$ with integer coefficients.  If $\gamma_i \ge 0$ then 
we say that the polynomial 
$f(t)$ is \emph{$\gamma$-positive}. It is known that $A_n(t)$ is $\gamma$-positive and that its 
coefficients $\gamma_i$ have a nice combinatorial interpretation. Indeed, let $\widehat \sym_n$ be 
the set of 
permutations in $\sym_n$ that have no two adjacent descents and no descent in 
the last position. In \cite{ShapiroWoanGetu1983}, Shapiro, Woan and Getu show that
$$\gamma_j=\{\sigma \in \widehat \sym_n\,\mid\, \des(\sigma)=j\}.$$
For example, $A_3(t)=(1+t)^2+2t$ with $\gamma_0=1$ and $\gamma_1=2$. The permutations in 
$\widehat \sym_3$ are, $123$ with no descents and; $213$ and $312$ with one descent.
Gal \cite{Gal2005} and Br\"and\'en \cite{Branden2004,Branden2008} have introduced the use of the 
$\gamma$-basis in different contexts. Gal 
conjectured that the $\gamma$-coefficients of the $h$-polynomial of a flag 
simple polytope are all nonnegative. In particular, $A_n(t)$ is the $h$-vector of the permutahedron 
that is a flag 
simple polytope so Gal's conjecture is confirmed in this 
case. Postnikov, Reiner and Williams \cite{PostnikovReinerWilliams2008} have confirmed Gal's 
conjecture for the family of chordal nestohedra that is a large family of flag simple polytopes. 
For more information about $\gamma$-positivity see \cite{Branden2014}.

We will show that the properties discussed above for the Eulerian polynomial $A_n(t)$ are also 
shared by the polynomial $T_n(t)$ in a similar fashion. The degree of $T_n(t)$ is also $n-1$ and it 
is 
easy to see from the definition 
of a descent that 
\begin{align}
|\T_{n,i}|=|\T_{n,n-1-i}|,
\end{align}
so $T_n(t)$ is also symmetric. 
Indeed there is a natural bijection $\T_{n,i} \simeq 
\T_{n,n-1-i}$ where the 
image of a labeled rooted tree $T \in \T_{n,i}$, is the tree in $\T_{n,n-1-i}$ with the same shape 
of $T$ but where each label $i$ has been replaced by $n+1-i$. For the example in Figure 
\ref{figure:rooted_trees_n3}, $T_3(t)=2+5t+2t^2$.

\begin{figure}[h]
 \begin{tikzpicture}[line join=bevel,scale=0.9]
\tikzstyle{every node}=[circle, draw,inner sep=1pt, minimum width=14pt,scale=0.8]
	\draw (2,5)  node (v3){1};
    \draw (1,4) node (v2){2};
    \draw (3,4) node (v1){3};
    \draw [color=blue,very thick] (v3) --  (v2) ;
    \draw [color=blue,very thick] (v3) --  (v1) ;

  \draw (2,3)  node (v3){1};
    \draw (2,2) node (v2){2};
    \draw (2,1) node (v1){3};
    \draw [color=blue,very thick] (v3) --  (v2) ;
    \draw [color=blue,very thick] (v2) --  (v1) ;

 \draw (7,5)  node (v3){2};
    \draw (6,4) node (v2){3};
    \draw (8,4) node (v1){1};
    \draw [color=blue,very thick] (v3) --  (v2) ;
    \draw [color=red,very thick] (v3) --  (v1) ;

  \draw (5.5,3)  node (v3){2};
    \draw (5.5,2) node (v2){1};
    \draw (5.5,1) node (v1){3};
    \draw [color=red,very thick] (v3) --  (v2) ;
    \draw [color=blue,very thick] (v2) --  (v1) ;

\draw (6.5,3)  node (v3){3};
    \draw (6.5,2) node (v2){1};
    \draw (6.5,1) node (v1){2};
    \draw [color=red,very thick] (v3) --  (v2) ;
    \draw [color=blue,very thick] (v2) --  (v1) ;

 \draw (7.5,3)  node (v3){1};
    \draw (7.5,2) node (v2){3};
    \draw (7.5,1) node (v1){2};
    \draw [color=blue,very thick] (v3) --  (v2) ;
    \draw [color=red,very thick] (v2) --  (v1) ;

 \draw (8.5,3)  node (v3){2};
    \draw (8.5,2) node (v2){3};
    \draw (8.5,1) node (v1){1};
    \draw [color=blue,very thick] (v3) --  (v2) ;
    \draw [color=red,very thick] (v2) --  (v1) ;
 
  \draw (12,5)  node (v3){3};
    \draw (11,4) node (v2){2};
    \draw (13,4) node (v1){1};
    \draw [color=red,very thick] (v3) --  (v2) ;
    \draw [color=red,very thick] (v3) --  (v1) ;

  \draw (12,3)  node (v3){3};
    \draw (12,2) node (v2){2};
    \draw (12,1) node (v1){1};
    \draw [color=red,very thick] (v3) --  (v2) ;
    \draw [color=red,very thick] (v2) --  (v1) ;
\tikzstyle{every node}=[inner sep=1pt, minimum width=14pt,scale=0.8]

\draw (2,0) node {$\des(T)=0$};
\draw (7,0) node {$\des(T)=1$};
\draw (12,0) node {$\des(T)=2$};

\end{tikzpicture}
\caption{All labeled rooted trees on $[3]$}
\label{figure:rooted_trees_n3}
\end{figure}
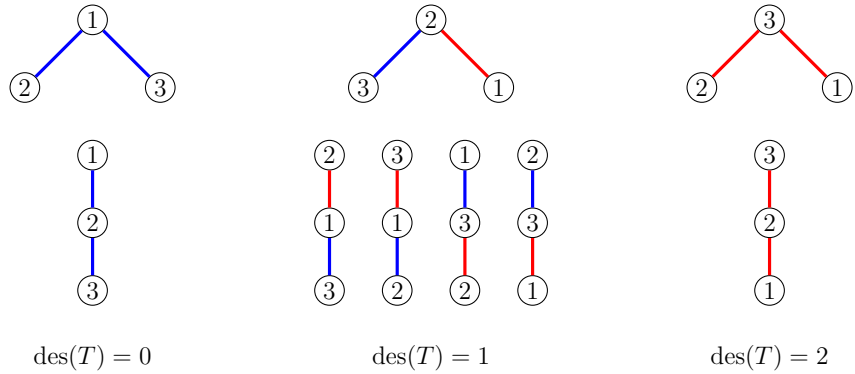 

In \cite{Drake2008} Drake proves the following nice product formula for $T_n(t)$.

\begin{theorem}[{\cite[Example 1.7.2]{Drake2008}}]
For $n \ge 1,$
\begin{align}  \label{equation:drake} \sum_{i = 0}^{n-1} |\T_{n,i}| t^i =
\prod_{i=1}^{n-1}((n-i)+it).
\end{align}
\end{theorem}

In particular, setting $t=1$ in (\ref{equation:drake}) reduces to the classical 
formula $|\T_n|=n^{n-1}$. Equation (\ref{equation:drake}) implies  that 
all the roots of this polynomial are 
real and negative. It is known and not difficult to show that a real-rooted symmetric polynomial 
with 
positive real coefficients is $\gamma$-positive (see \cite{Branden2014,Gal2005}). For example 
$T_3(t)=2(1+t)^2+t$, so $\gamma_0=2$ and $\gamma_1=1$.
Although $T_n(t)$ is not in general an $h$-vector of a convex polytope (for example $h_0 \neq 1$ 
for $n 
\ge 3$), 
it is of interest to find combinatorial formulas and interpretations of positive 
$\gamma$-coefficients of general symmetric polynomials. Let $\gamma_j(T_n(t))$ denote the 
 $j$-th gamma coefficient of the symmetric polynomial $T_n(t)$.

Equation (\ref{equation:drake}) can be used to find a formula for 
the coefficients $\gamma_j$ of $T_n(t)$.

\begin{theorem}\label{theorem:gammaformulas} For $n \ge 1$,
 \begin{equation}\label{equation:gammaformulas}
 \displaystyle
 \gamma_j(T_n(t))= \begin{cases}
   \displaystyle \sum_{\substack{J  \subset [\frac{n-1}{2}]\\|J|=j}} \prod_{i \in J}(n-2i)^2
\prod_{s \in [\frac{n-1}{2}] \setminus J} s(n-s)& \text{ if $n$ is odd}\\
    \displaystyle \frac{n}{2}\sum_{\substack{J  \subset [\frac{n-2}{2}]\\|J|=j}} \prod_{i \in
J}(n-2i)^2 \prod_{s \in [\frac{n-2}{2}] \setminus J} s(n-s) & \text{ if $n$ is even}\\
  \end{cases}
\end{equation}
\end{theorem}

\begin{proof}
 If we multiply
 \begin{align*}
 (n-i+it)(i+(n-i)t)&=(n-i)i+[(n-i)^2+i^2]t+(n-i)it^2\\
 &=(n-i)i(1+t^2)+[(n-i)^2+i^2]t\\
 &=(n-i)i(1+2t+t^2) + [(n-i)^2-2i(n-i)+i^2]t\\
 &=(n-i)i(1+t)^2+(n-2i)^2t.
 \end{align*}
Equation \ref{equation:drake} can be written as
\[\displaystyle
 \T(t)= \begin{cases}
   \displaystyle \prod_{i=1}^{\frac{n-1}{2}}[(n-i)i(1+t)^2+(n-2i)^2t] &  \text{ if $n$ is
odd,}\\
    \displaystyle \frac{n}{2}(1+t)\prod_{i=1}^{\frac{n-2}{2}}[(n-i)i(1+t)^2+(n-2i)^2t] &
\text{ if $n$ is even,}\\ \end{cases}
\]
implying Formula \ref{equation:gammaformulas}.
\end{proof}

The purpose of this note is to present four different combinatorial interpretations for the 
coefficients $\gamma_j$ that are consequences of results in the work of the
author and Wachs in \cite{DleonWachs2013} and of the author in \cite{Dleon2014}. We present now one 
of these combinatorial interpretations, whose proof will be given in Section 
\ref{section:binarytrees}.

A \emph{planar leaf-labeled binary tree} with label set $[n]$ is a rooted tree (a priori 
without 
labels) in which the 
set of children of every internal node is a totally ordered set with exactly two elements (the 
left and right children) and where 
each leaf has been asigned a unique element from the set $[n]$. 
By a subtree in a rooted tree $T$ we mean the rooted tree induced by the descendents of any node 
$x$ of $T$, including and rooted at $x$.
We say that a planar leaf-labeled 
binary tree with label set $[n]$ is 
\emph{normalized} 
if in each subtree, the leftmost leaf is the one with the smallest label. We denote 
the set of 
normalized binary trees with label set $[n]$ by $\nor_n$. All 
normalized trees with leaf labels in $[3]$ are illustrated in Figure 
\ref{figure:normalizedn3}. 

A \emph{right descent} in a normalized tree is an internal node that is the right child of its 
parent. For $T \in \nor_n$ we define $\rdes(T):=|\{\text{right descents of T}\}|$. A \emph{double 
right descent} is a right descent whose parent is also a right 
descent. We denote by $\ndrd_n$ the set of trees in $\nor_n$ with no double right descents. 

\begin{theorem}\label{theorem:gammacoefficientsrightdescents}
For $n\ge 1$ and $j \in \{0,1,\cdots,\lfloor  \frac{n-1}{2} \rfloor\}$,
\begin{align*}
 \gamma_j(T_n(t)) = |\{T \in \ndrd_n\,\mid\, \rdes(T)=j\}|.
\end{align*}
\end{theorem}

As it is illustrated in Figure \ref{figure:normalizedn3}, there are two trees in $\ndrd_3$ (for 
$n=3$ it happens to be equal to $\nor_3$) with $\rdes(T)=0$ and one with $\rdes(T)=1$, 
corresponding to $\gamma_0=2$ and $\gamma_1=1$ respectively.

\begin{figure}[ht]
  \begin{tikzpicture}[scale=0.7]

\draw[dotted, very thick] (-1,-2.5) -- (-1,1.5) -- (4,1.5) --(4,-2.5)-- cycle;
\draw[dotted, very thick] (4.5,-2.5) -- (4.5,1.5) -- (12.5,1.5) --(12.5,-2.5)-- cycle;
\draw (1.5,-2) node {\tiny $\rdes(T)=1$};
\draw (9,-2) node {\tiny $\rdes(T)=0$};

\begin{scope}[xshift=9cm, yshift=-1cm]

\tikzstyle{every node}=[draw,inner sep=1mm,scale=1]
    \draw [circle] (1,1)  node (i1){$$};
    \draw [circle] (2,2)  node (i2){$$};

\tikzstyle{every node}=[inner sep=1pt, minimum width=14pt,scale=1]

    \draw (0,0)  node (m){$1$};
    \draw (2,0)  node (l1){$2$};
    \draw (3,1)  node (l2){$3$};

    \draw (m) --  (i1) ;
    \draw (i1) --  (l1) ;
    \draw (i1) --  (i2) ;
    \draw (i2) --  (l2) ;
\end{scope}

\begin{scope}[xshift=5cm,yshift=-1cm]

\tikzstyle{every node}=[draw,inner sep=1mm,scale=1]
    \draw [circle] (1,1)  node (i1){$$};
    \draw [circle] (2,2)  node (i2){$$};

\tikzstyle{every node}=[inner sep=1pt, minimum width=14pt,scale=1]

    \draw (0,0)  node (m){$1$};
    \draw (2,0)  node (l1){$3$};
    \draw (3,1)  node (l2){$2$};

    \draw (m) --  (i1) ;
    \draw (i1) --  (l1) ;
    \draw (i1) --  (i2) ;
    \draw (i2) --  (l2) ;
\end{scope}
\begin{scope}[xshift=-1.5cm,yshift=-1cm]

\tikzstyle{every node}=[draw,inner sep=1mm,scale=1]
    \draw [circle] (3,1)  node (i1){$$};
    \draw [circle] (2,2)  node (i2){$$};

\tikzstyle{every node}=[inner sep=1pt, minimum width=14pt,scale=1]

    \draw (2,0)  node (m){$2$};
    \draw (4,0)  node (l1){$3$};
    \draw (1,1)  node (l2){$1$};

    \draw (m) --  (i1) ;
    \draw (i1) --  (l1) ;
    \draw (i1) --  (i2) ;
    \draw (i2) --  (l2) ;
\end{scope}

\end{tikzpicture}
 \caption{Set of normalized trees in $\ndrd_3=\nor_3$}
 \label{figure:normalizedn3}
\end{figure}
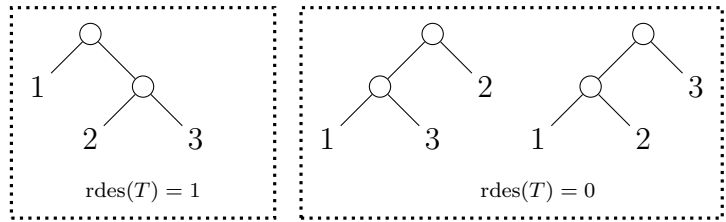

In Section \ref{section:binarytrees} we provide the proof of Theorem 
\ref{theorem:gammacoefficientsrightdescents} and an additional version of Theorem 
\ref{theorem:gammacoefficientsrightdescents} also in terms of normalized trees but with a statistic 
different than $\rdes$. In Section \ref{section:stirlingpermutations} we provide two additional 
versions of Theorem \ref{theorem:gammacoefficientsrightdescents} in terms of the Stirling 
permutations introduced by Gessel and Stanley in \cite{StanleyGessel1978}. In Section 
\ref{section:epositivity} we discuss a generalization of the $\gamma$-positivity of $T_n(t)$ to 
the positivity of certain symmetric function in the basis of elementary symmetric functions.

\section{Combinatorial interpretations in terms of binary trees}\label{section:binarytrees}

Now we consider normalized trees $T$ where every internal node $x$ of $T$ has been assigned an 
element $\clr(x) \in \{0,1\}$. We call an element of this set of trees a \emph{bicolored 
normalized tree} on $[n]$. A \emph{bicolored comb} is a bicolored normalized tree $T$ satisfying 
the following coloring restriction:

\begin{enumerate}
 \item[(C)] If $x$ is a right descent of $T$ then $\clr(x)=0$ and $\clr(\p(x))=1$.
\end{enumerate}

We denote by $\comb_n$ the set of bicolored combs and by $\comb_{n,i}$ the set of bicolored combs 
where $i$ internal nodes have been colored $1$ (and $n-1-i$ colored $0$). Figure 
\ref{figure:bicoloredcombsn3} illustrates the bicolored combs on $[3]$ grouped by the number of 
internal nodes that have been colored $1$.

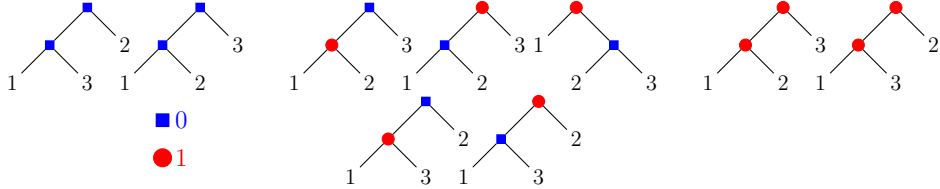
\begin{figure}[ht]
  \begin{tikzpicture}[scale=0.5]
\tikzstyle{every node}=[fill, draw,inner sep=3pt,scale=0.8]
    \draw [color=blue] (1,-1)  node (i1){};
    \draw [circle, color=red] (1,-2)  node (i2){};
\tikzstyle{every node}=[inner sep=3pt,scale=0.8]
	\draw [color=blue] (1.5,-1)  node (i1){0};
    \draw [circle, color=red] (1.5,-2)  node (i2){$1$};

\begin{scope}[xshift=-3cm]

\tikzstyle{every node}=[fill, draw,inner sep=2pt,scale=0.8]
    \draw [color=blue] (1,1)  node (i1){};
    \draw [color=blue] (2,2)  node (i2){};

\tikzstyle{every node}=[inner sep=1pt, minimum width=14pt,scale=0.7]

    \draw (0,0)  node (m){$1$};
    \draw (2,0)  node (l1){$3$};
    \draw (3,1)  node (l2){$2$};

    \draw (m) --  (i1) ;
    \draw (i1) --  (l1) ;
    \draw (i1) --  (i2) ;
    \draw (i2) --  (l2) ;
\end{scope}
\begin{scope}[xshift=0]
\tikzstyle{every node}=[draw,circle,inner sep=1mm,scale=0.8]

\tikzstyle{every node}=[fill, draw,inner sep=2pt,scale=0.8]
    \draw [color=blue] (1,1)  node (i1){};
    \draw [color=blue] (2,2)  node (i2){};

\tikzstyle{every node}=[inner sep=1pt, minimum width=14pt,scale=0.7]

    \draw (0,0)  node (m){$1$};
    \draw (2,0)  node (l1){$2$};
    \draw (3,1)  node (l2){$3$};

    \draw (m) --  (i1) ;
    \draw (i1) --  (l1) ;
    \draw (i1) --  (i2) ;
    \draw (i2) --  (l2) ;
\end{scope}

\begin{scope}[xshift=4.5cm,yshift=0]
\tikzstyle{every node}=[fill, draw,inner sep=2pt,scale=0.8]
    \draw [circle,color=red] (1,1)  node (i1){};
    \draw [color=blue] (2,2)  node (i2){};

\tikzstyle{every node}=[inner sep=1pt, minimum width=14pt,scale=0.7]

    \draw (0,0)  node (m){$1$};
    \draw (2,0)  node (l1){$2$};
    \draw (3,1)  node (l2){$3$};

    \draw (m) --  (i1) ;
    \draw (i1) --  (l1) ;
    \draw (i1) --  (i2) ;
    \draw (i2) --  (l2) ;
\end{scope}

\begin{scope}[xshift=7.5cm,yshift=0]
\tikzstyle{every node}=[fill, draw,inner sep=2pt,scale=0.8]
    \draw [color=blue] (1,1)  node (i1){};
    \draw [circle,color=red] (2,2)  node (i2){};

\tikzstyle{every node}=[inner sep=1pt, minimum width=14pt,scale=0.7]

    \draw (0,0)  node (m){$1$};
    \draw (2,0)  node (l1){$2$};
    \draw (3,1)  node (l2){$3$};

    \draw (m) --  (i1) ;
    \draw (i1) --  (l1) ;
    \draw (i1) --  (i2) ;
    \draw (i2) --  (l2) ;
\end{scope}

\begin{scope}[xshift=6cm,yshift=-2.5cm]
\tikzstyle{every node}=[fill, draw,inner sep=2pt,scale=0.8]
    \draw [circle,color=red] (1,1)  node (i1){};
    \draw [color=blue] (2,2)  node (i2){};

\tikzstyle{every node}=[inner sep=1pt, minimum width=14pt,scale=0.7]

    \draw (0,0)  node (m){$1$};
    \draw (2,0)  node (l1){$3$};
    \draw (3,1)  node (l2){$2$};

    \draw (m) --  (i1) ;
    \draw (i1) --  (l1) ;
    \draw (i1) --  (i2) ;
    \draw (i2) --  (l2) ;
\end{scope}

\begin{scope}[xshift=9cm,yshift=-2.5cm]
\tikzstyle{every node}=[fill, draw,inner sep=2pt,scale=0.8]
    \draw [color=blue] (1,1)  node (i1){};
    \draw [circle,color=red] (2,2)  node (i2){};

\tikzstyle{every node}=[inner sep=1pt, minimum width=14pt,scale=0.7]

    \draw (0,0)  node (m){$1$};
    \draw (2,0)  node (l1){$3$};
    \draw (3,1)  node (l2){$2$};

    \draw (m) --  (i1) ;
    \draw (i1) --  (l1) ;
    \draw (i1) --  (i2) ;
    \draw (i2) --  (l2) ;
\end{scope}

\begin{scope}[xshift=10cm]
\tikzstyle{every node}=[fill, draw,inner sep=2pt,scale=0.8]
    \draw [color=blue] (3,1)  node (i1){};
    \draw [circle,color=red] (2,2)  node (i2){};

\tikzstyle{every node}=[inner sep=1pt, minimum width=14pt,scale=0.7]

    \draw (2,0)  node (m){$2$};
    \draw (4,0)  node (l1){$3$};
    \draw (1,1)  node (l2){$1$};
    
    \draw (m) --  (i1) ;
    \draw (i1) --  (l1) ;
    \draw (i1) --  (i2) ;
    \draw (i2) --  (l2) ;
\end{scope}

\begin{scope}[xshift=18.5cm]
\tikzstyle{every node}=[fill, draw,inner sep=2pt,scale=0.8]
    \draw [circle,color=red] (1,1)  node (i1){};
    \draw [circle,color=red] (2,2)  node (i2){};

\tikzstyle{every node}=[inner sep=1pt, minimum width=14pt,scale=0.7]

    \draw (0,0)  node (m){$1$};
    \draw (2,0)  node (l1){$3$};
    \draw (3,1)  node (l2){$2$};
    
    \draw (m) --  (i1) ;
    \draw (i1) --  (l1) ;
    \draw (i1) --  (i2) ;
    \draw (i2) --  (l2) ;
\end{scope}
\begin{scope}[xshift=15.5cm]
\tikzstyle{every node}=[fill, draw,inner sep=2pt,scale=0.8]
    \draw [circle,color=red] (1,1)  node (i1){};
    \draw [circle,color=red] (2,2)  node (i2){};

\tikzstyle{every node}=[inner sep=1pt, minimum width=14pt,scale=0.7]

    \draw (0,0)  node (m){$1$};
    \draw (2,0)  node (l1){$2$};
    \draw (3,1)  node (l2){$3$};

    \draw (m) --  (i1) ;
    \draw (i1) --  (l1) ;
    \draw (i1) --  (i2) ;
    \draw (i2) --  (l2) ;
\end{scope}

\end{tikzpicture}
 \caption{Set of bicolored combs on $[3]$}
 \label{figure:bicoloredcombsn3}
\end{figure}

Denote by $\widetilde T \in \nor_n$ the underlying uncolored normalized tree associated to a tree $T 
\in \comb_n$. Note that the coloring condition (C) implies that $\widetilde T \in 
\ndrd_n$. Indeed, in a 
double right descent the coloring condition (C) cannot be satisfied since the parent of a double 
right descent is also a right descent.
Also note that the 
monochromatic combs in $\comb_{n,0}$ and
$\comb_{n,n-1}$ are just the traditional left combs that are described in \cite{Wachs1998} and 
that index a basis for the space $\lie(n)$, the multilinear component  of the free Lie 
algebra over $\CC$ on $n$ generators (see \cite{Wachs1998} for details).
Liu \cite{Liu2010} and 
Dotsenko-Khoroshkin \cite{DotsenkoKhoroshkin2007} independently proved a conjecture of Feigin 
regarding the dimension of the multilinear component $\lie_2(n)$ of the free Lie algebra with two 
compatible brackets, a generalization of $\lie(n)$. 

\begin{theorem}[{\cite{DotsenkoKhoroshkin2007,Liu2010}}]
 For $n\ge 1$, $\dim \lie_2(n)=|\T_{n}|.$
\end{theorem}

In particular, the space $\lie_2(n)$ has the decomposition
\begin{align*}
 \lie_2(n)=\bigoplus_{i=0}^{n-1}\lie_2(n,i),
\end{align*}

where the subspace $\lie_2(n,i)$ is the component generated by certain ``bracketed permutations" 
with exactly $i$ brackets of one of the types. Liu finds the following formula 
for the dimension of $\lie_2(n,i)$.

\begin{theorem}[{\cite[Proposition 11.3]{Liu2010}}] For $n\ge 1$ and $i \in \{0,1,\cdots,n-1\}$,
 \begin{align*}
  \dim \lie_2(n,i)=|\T_{n,i}|.
 \end{align*}
\end{theorem}

In \cite{DleonWachs2013} the author and Wachs studied the relation between $\lie_2(n,i)$ and the 
cohomology of the maximal intervals of a poset of weighted partitions.  Using poset topology 
techniques they found the following alternative description for the dimension of $\lie_2(n,i)$.

\begin{theorem}[{\cite[Section 5]{DleonWachs2013}}]For $n\ge 1$ and $i \in \{0,1,\cdots,n-1\}$,
 \begin{align*}
  \dim \lie_2(n,i)=|\comb_{n,i}|.
 \end{align*}
\end{theorem}

\begin{corollary}\label{corollary:comb_tree_equicardinal} For every $n \ge 1$ and $i \in 
\{0,\cdots, n-1\}$,
 \begin{align*}
  |\comb_{n,i}|=|\T_{n,i}|.
 \end{align*}
\end{corollary}

\begin{problem}
 Find an explicit bijection $\comb_{n,i} \rightarrow \T_{n,i}$, for every $n \ge 
1$ and $i \in \{0,\cdots, n-1\}$.
\end{problem}

\begin{theorem}[{Theorem \ref{theorem:gammacoefficientsrightdescents}}]
For $n\ge 1$ and $j \in \{0,1,\cdots,\lfloor  \frac{n-1}{2} \rfloor\}$,
\begin{align*}
 \gamma_j(T_n(t)) = |\{T \in \ndrd_n\,\mid\, \rdes(T)=j\}|.
\end{align*}
\end{theorem}
\begin{proof}
First note that by the comments above if $T\in \comb_n$ then its underlying uncolored 
tree $\widetilde T \in \ndrd_n$.

For a tree $T \in \comb_{n}$ call $\free(T)$ the number of internal nodes that are not right 
descents and whose 
right child is a leaf.  Then in $\ndrd_n$ we have that 
\[
\free(T)+2 \rdes(T)=n-1.
\]
Over the set of bicolored combs with $m$ free nodes there is a free action of $(\ZZ_2)^m$ by 
toggling
the colors of the free nodes. Then there are $2^m$ bicolored combs with the same underlying 
tree $\widetilde T \in \ndrd_n$.
By Corollary \ref{corollary:comb_tree_equicardinal} we can write $T_n(t)$ as 

\begin{align*}
T_n(t)&=\sum_{i=0}^{n-1}|\T_{n,i}|t^i\\
&=\sum_{i=0}^{n-1}|\comb_{n,i}|t^i\\
&=\sum_{T \in \comb_n} t^{|\{x \text{ internal in } T \,\mid\, \clr(x)=1 \}|}\\
&=\sum_{\mathfrak{T} \in \ndrd_n}\sum_{\substack{T \in \comb_n\\ \widetilde T =\mathfrak{T}  }} 
t^{|\{x \text{ internal in } T \,\mid\, \clr(x)=1 \}|}\\
&=\sum_{\mathfrak{T} \in \ndrd_n}t^{\rdes(\mathfrak{T})}(1+t)^{\free(\mathfrak{T})}\\
&=\sum_{\mathfrak{T} \in \ndrd_n}t^{\rdes(\mathfrak{T})}(1+t)^{n-1-2\rdes(\mathfrak{T})}.\qedhere
\end{align*}
\end{proof}

\subsection{A second description in terms of normalized trees}
Define now the \emph{valency} $v(x)$ of a node (internal or leaf) $x$ of $T \in \nor_n$ to be the 
minimal label in the subtree of $T$ rooted at $x$. For an internal node $x$ of $T$ let $L(x)$ and 
$R(x)$ denote the left and right children of $x$ respectively. A \emph{Lyndon node} is an internal 
node $x$ of $T$ such that 
\begin{align}\label{equation:lyndoncondition}
 v(R(L(x)))>v(R(x)).
\end{align}
A \emph{Lyndon tree} is a normalized tree in which all its internal nodes are Lyndon. We denote 
$\nlyn(T)$ the number of non-Lyndon nodes in $T$. A \emph{double non-Lyndon node} is a non-Lyndon 
node that is the left child of its parent and its parent is also a non-Lyndon node. We denote the 
set of trees in $\nor_n$ with no 
double non-Lyndon nodes by $\ndnl_n$. A \emph{bicolored Lyndon tree} is a bicolored normalized tree
satisfying the coloring condition:
\begin{enumerate}
 \item[(L)] For every non-Lyndon node $x$ of $T$ 
then $\clr(x)=0$ and $\clr(L(x))=1$.
\end{enumerate}

The set of bicolored Lyndon trees is denoted $\lyn_n$ and the set of the ones with exactly $i$ 
nodes with color $1$ is denoted $\lyn_{n,i}$. 

\begin{figure}[ht]
  \begin{tikzpicture}[scale=0.5]

\tikzstyle{every node}=[fill, draw,inner sep=3pt,scale=0.8]
    \draw [color=blue] (1,-1)  node (i1){};
    \draw [circle, color=red] (1,-2)  node (i2){};
\tikzstyle{every node}=[inner sep=3pt,scale=0.8]
	\draw [color=blue] (1.5,-1)  node (i1){0};
    \draw [circle, color=red] (1.5,-2)  node (i2){$1$};

\begin{scope}[xshift=-2.5cm]

\tikzstyle{every node}=[fill, draw,inner sep=2pt,scale=0.8]
    \draw [color=blue] (1,1)  node (i1){};
    \draw [color=blue] (2,2)  node (i2){};

\tikzstyle{every node}=[inner sep=1pt, minimum width=14pt,scale=0.7]

    \draw (0,0)  node (m){$1$};
    \draw (2,0)  node (l1){$3$};
    \draw (3,1)  node (l2){$2$};

    \draw (m) --  (i1) ;
    \draw (i1) --  (l1) ;
    \draw (i1) --  (i2) ;
    \draw (i2) --  (l2) ;
\end{scope}

\begin{scope}[xshift=0cm]

\tikzstyle{every node}=[fill, draw,inner sep=2pt,scale=0.8]
    \draw [color=blue] (3,1)  node (i1){};
    \draw [color=blue] (2,2)  node (i2){};

\tikzstyle{every node}=[inner sep=1pt, minimum width=14pt,scale=0.7]

    \draw (2,0)  node (m){$2$};
    \draw (4,0)  node (l1){$3$};
    \draw (1,1)  node (l2){$1$};
    
    \draw (m) --  (i1) ;
    \draw (i1) --  (l1) ;
    \draw (i1) --  (i2) ;
    \draw (i2) --  (l2) ;
\end{scope}

\begin{scope}[xshift=4cm]

\tikzstyle{every node}=[fill, draw,inner sep=2pt,scale=0.8]
    \draw [circle,color=red] (3,1)  node (i1){};
    \draw [color=blue] (2,2)  node (i2){};

\tikzstyle{every node}=[inner sep=1pt, minimum width=14pt,scale=0.7]

    \draw (2,0)  node (m){$2$};
    \draw (4,0)  node (l1){$3$};
    \draw (1,1)  node (l2){$1$};
    
    \draw (m) --  (i1) ;
    \draw (i1) --  (l1) ;
    \draw (i1) --  (i2) ;
    \draw (i2) --  (l2) ;
\end{scope}

\begin{scope}[xshift=6.5cm]

\tikzstyle{every node}=[fill, draw,inner sep=2pt,scale=0.8]
    \draw [color=blue] (3,1)  node (i1){};
    \draw [circle,color=red] (2,2)  node (i2){};

\tikzstyle{every node}=[inner sep=1pt, minimum width=14pt,scale=0.7]

    \draw (2,0)  node (m){$2$};
    \draw (4,0)  node (l1){$3$};
    \draw (1,1)  node (l2){$1$};
    
    \draw (m) --  (i1) ;
    \draw (i1) --  (l1) ;
    \draw (i1) --  (i2) ;
    \draw (i2) --  (l2) ;
\end{scope}

\begin{scope}[xshift=11cm]
\tikzstyle{every node}=[fill, draw,inner sep=2pt,scale=0.8]
    \draw [circle,color=red] (1,1)  node (i1){};
    \draw [color=blue] (2,2)  node (i2){};

\tikzstyle{every node}=[inner sep=1pt, minimum width=14pt,scale=0.7]

    \draw (0,0)  node (m){$1$};
    \draw (2,0)  node (l1){$2$};
    \draw (3,1)  node (l2){$3$};

    \draw (m) --  (i1) ;
    \draw (i1) --  (l1) ;
    \draw (i1) --  (i2) ;
    \draw (i2) --  (l2) ;

\end{scope}

\begin{scope}[xshift=6.5cm,yshift=-2.5cm]
\tikzstyle{every node}=[fill, draw,inner sep=2pt,scale=0.8]
    \draw [circle,color=red] (1,1)  node (i1){};
    \draw [color=blue] (2,2)  node (i2){};

\tikzstyle{every node}=[inner sep=1pt, minimum width=14pt,scale=0.7]

    \draw (0,0)  node (m){$1$};
    \draw (2,0)  node (l1){$3$};
    \draw (3,1)  node (l2){$2$};

    \draw (m) --  (i1) ;
    \draw (i1) --  (l1) ;
    \draw (i1) --  (i2) ;
    \draw (i2) --  (l2) ;
\end{scope}

\begin{scope}[xshift=9.5cm,yshift=-2.5cm]
\tikzstyle{every node}=[fill, draw,inner sep=2pt,scale=0.8]
    \draw [color=blue] (1,1)  node (i1){};
    \draw [circle,color=red] (2,2)  node (i2){};

\tikzstyle{every node}=[inner sep=1pt, minimum width=14pt,scale=0.7]

    \draw (0,0)  node (m){$1$};
    \draw (2,0)  node (l1){$3$};
    \draw (3,1)  node (l2){$2$};

    \draw (m) --  (i1) ;
    \draw (i1) --  (l1) ;
    \draw (i1) --  (i2) ;
    \draw (i2) --  (l2) ;
\end{scope}
\begin{scope}[xshift=15.5cm]
\tikzstyle{every node}=[fill, draw,inner sep=2pt,scale=0.8]
    \draw [circle,color=red] (1,1)  node (i1){};
    \draw [circle,color=red] (2,2)  node (i2){};

\tikzstyle{every node}=[inner sep=1pt, minimum width=14pt,scale=0.7]

    \draw (0,0)  node (m){$1$};
    \draw (2,0)  node (l1){$3$};
    \draw (3,1)  node (l2){$2$};
    
    \draw (m) --  (i1) ;
    \draw (i1) --  (l1) ;
    \draw (i1) --  (i2) ;
    \draw (i2) --  (l2) ;
\end{scope}

\begin{scope}[xshift=18cm]

\tikzstyle{every node}=[fill, draw,inner sep=2pt,scale=0.8]
    \draw [circle,color=red] (3,1)  node (i1){};
    \draw [circle,color=red] (2,2)  node (i2){};

\tikzstyle{every node}=[inner sep=1pt, minimum width=14pt,scale=0.7]

    \draw (2,0)  node (m){$2$};
    \draw (4,0)  node (l1){$3$};
    \draw (1,1)  node (l2){$1$};
    
    \draw (m) --  (i1) ;
    \draw (i1) --  (l1) ;
    \draw (i1) --  (i2) ;
    \draw (i2) --  (l2) ;
\end{scope}
\end{tikzpicture}
 \caption{Set of bicolored Lyndon trees on $[3]$}
 \label{figure:bicoloredlyndontreesn3}
\end{figure}
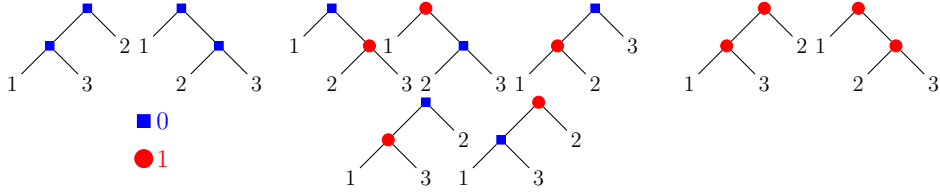

\begin{theorem}[{\cite[Section 5]{DleonWachs2013}}]
For every $n \ge 1$ and $i \in 
\{0,\cdots, n-1\}$,
 \begin{align*}
  \dim \lie_2(n,i)=|\lyn_{n,i}|.
 \end{align*}
 Hence, 
 \begin{align*}
  |\lyn_{n,i}|=|\T_{n,i}|.
 \end{align*}
\end{theorem}

The proof of the following theorem follows the same arguments of the proof of Theorem 
\ref{theorem:gammacoefficientsrightdescents}.

\begin{theorem}
For $n\ge 1$ and $j \in \{0,1,\cdots,\lfloor  \frac{n-1}{2} \rfloor\}$,
\begin{align*}
 \gamma_j(T_n(t)) = |\{T \in \ndnl_n\,\mid\, \nlyn(T)=j\}|.
\end{align*}
\end{theorem}

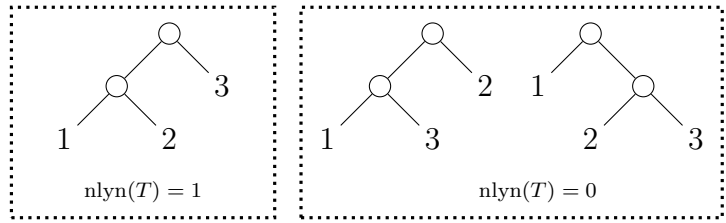
\begin{figure}[ht]
  \begin{tikzpicture}[scale=0.7]

\draw[dotted, very thick] (-1,-2.5) -- (-1,1.5) -- (4,1.5) --(4,-2.5)-- cycle;
\draw[dotted, very thick] (4.5,-2.5) -- (4.5,1.5) -- (12.5,1.5) --(12.5,-2.5)-- cycle;
\draw (1.5,-2) node {\tiny $\nlyn(T)=1$};
\draw (9,-2) node {\tiny $\nlyn(T)=0$};

\begin{scope}[xshift=0cm, yshift=-1cm]

\tikzstyle{every node}=[draw,inner sep=1mm,scale=1]
    \draw [circle] (1,1)  node (i1){$$};
    \draw [circle] (2,2)  node (i2){$$};

\tikzstyle{every node}=[inner sep=1pt, minimum width=14pt,scale=1]

    \draw (0,0)  node (m){$1$};
    \draw (2,0)  node (l1){$2$};
    \draw (3,1)  node (l2){$3$};

    \draw (m) --  (i1) ;
    \draw (i1) --  (l1) ;
    \draw (i1) --  (i2) ;
    \draw (i2) --  (l2) ;
\end{scope}

\begin{scope}[xshift=5cm,yshift=-1cm]

\tikzstyle{every node}=[draw,inner sep=1mm,scale=1]
    \draw [circle] (1,1)  node (i1){$$};
    \draw [circle] (2,2)  node (i2){$$};

\tikzstyle{every node}=[inner sep=1pt, minimum width=14pt,scale=1]

    \draw (0,0)  node (m){$1$};
    \draw (2,0)  node (l1){$3$};
    \draw (3,1)  node (l2){$2$};

    \draw (m) --  (i1) ;
    \draw (i1) --  (l1) ;
    \draw (i1) --  (i2) ;
    \draw (i2) --  (l2) ;
\end{scope}
\begin{scope}[xshift=8cm,yshift=-1cm]

\tikzstyle{every node}=[draw,inner sep=1mm,scale=1]
    \draw [circle] (3,1)  node (i1){$$};
    \draw [circle] (2,2)  node (i2){$$};

\tikzstyle{every node}=[inner sep=1pt, minimum width=14pt,scale=1]

    \draw (2,0)  node (m){$2$};
    \draw (4,0)  node (l1){$3$};
    \draw (1,1)  node (l2){$1$};

    \draw (m) --  (i1) ;
    \draw (i1) --  (l1) ;
    \draw (i1) --  (i2) ;
    \draw (i2) --  (l2) ;
\end{scope}

\end{tikzpicture}
 \caption{Set of normalized trees in $\ndnl_3=\nor_3$}
 \label{figure:normalizedlyndontypen3}
\end{figure}

\section{Combinatorial interpretation in terms of Stirling 
permutations}\label{section:stirlingpermutations}

Consider now the set of multipermutations of the multiset $\{1,1,2,2,\allowbreak\cdots,n,n\}$ such 
that all 
numbers between the two occurrences of any number $m$ are larger than $m$.
To this family belongs 
for example the permutation $12234431$ but not $11322344$ since $2$ is less than $3$ and $2$ is 
between the two occurrences of $3$. This family (denoted $\Q_n$) of permutations was introduced by 
Gessel and Stanley in \cite{StanleyGessel1978} and the permutations in $\Q_n$ are known as 
\emph{Stirling Permutations}. 

For a permutation $\theta=\theta_1 \theta_2 \dots \theta_{2n}$ in $\Q_n$ we say that the position 
$i$ contains a \emph{first occurrence} of a letter if $\theta_j\neq\theta_i$ for all $j < i$, 
otherwise we say that it contains a \emph{second occurrence}.
An \emph{ascending adjacent pair} in $\theta$ is a pair $(a,b)$ such that $a 
< b$ and in $\theta$ the second occurrence of $a$ is the immediate predecessor of the first 
occurrence of $b$.
An \emph{ascending adjacent sequence (of length $2$)} is a sequence $a<b<c$ such that $(a,b)$ and 
$(b,c)$ are both ascending adjacent pairs. For example, in 
$\theta=13344155688776$ the ascending adjacent pairs are $(1,5),\,(5,6)$ and $(3,4)$ but the only 
ascending adjacent sequence is $1<5<6$. We denote $\naas_n$ the 
set of all Stirling permutations in 
$\Q_n$ that do not contain ascending adjacent sequences.
Similarly, a \emph{terminally nested pair} in $\theta$ is a pair $(a,b)$ such that $a 
< b$ and in $\theta$ the second occurrence of $a$ is the immediate successor of the second 
occurrence of $b$.
A \emph{terminally nested sequence (of length $2$)} is a sequence $a<b<c$ such that $(a,b)$ and 
$(b,c)$ are both terminally nested pairs. For example, in 
$\theta=13443566518877$ the terminally nested pairs are $(1,5),\,(5,6)$ and $(3,4)$ but the only 
terminally nested sequence is $1<5<6$. We denote $\ntns_n$ the 
set of all Stirling permutations in $\Q_n$ that do not contain terminally nested sequences.
For $\sigma \in \Q_n$, we denote $\aapair(\sigma)$ the number of ascending adjacent pairs in 
$\sigma$ and $\tnpair(\sigma)$ the number of terminally nested pairs in $\sigma$.

The following result in \cite{Dleon2014} relates the statistics above in $\Q_{n-1}$ with the ones 
previously discussed for $\nor_n$.
\begin{proposition}[{\cite[Proposition 4.8]{Dleon2014}}]
 There is a bijection $\phi: \nor_n \rightarrow \Q_{n-1}$ such that for every $T \in \nor_n$,
 \begin{enumerate}
  \item $\rdes(T)$=$\tnpair(\phi(T))$
  \item $\nlyn(T)$=$\aapair(\phi(T))$
  \item $\phi(\ndrd_n)=\naas_{n-1}$
  \item $\phi(\ndnl_n)=\ntns_{n-1}$.
 \end{enumerate}
\end{proposition}

\begin{corollary}
For $n\ge 1$ and $j \in \{0,1,\cdots,\lfloor  \frac{n-1}{2} \rfloor\}$,
\begin{align*}
 \gamma_j(T_n(t)) =& |\{T \in \ntns_{n-1}\,\mid\, \tnpair(T)=j\}|\\
 =&|\{T \in \naas_{n-1}\,\mid\, \aapair(T)=j\}|.
\end{align*}
\end{corollary}

\begin{example}The Stirling permutations in $\Q_2$ are $1122$, $1221$ and $2211$. In this 
particular case $\Q_2=\naas_2=\ntns_2$ and the statistics in Table 
\ref{table:stirlingpermutationsn2} imply that $\gamma_0=2$ and $\gamma_1=1$ are the 
$\gamma$-coefficients of the polynomial $\T_3(t)$.
\begin{table}[ht]

 \begin{center}
\begin{tabular}{|c|c|c|}\hline
\textbf{$\sigma$} & \textbf{$\tnpair$} & \textbf{$\aapair$}\\\hline
$1122$ & 0 & 1\\\hline
$1221$ & 1 & 0\\\hline
$2211$ & 0 & 0\\\hline
 \end{tabular}
 \end{center}
\caption{$\tnpair$ and $\aapair$ statistics in $\Q_2$.}
\label{table:stirlingpermutationsn2}
\end{table}
\end{example}

\section{A comment about $\gamma$-positivity and $e$-positivity}\label{section:epositivity}
Let $\xx= 
x_1,x_2,\dots$ be an infinite set of variables and $\Lambda=\Lambda_{\QQ}$ the ring of symmetric 
functions with rational coefficients on the variables $\xx$, that is, the ring of power series on 
$\xx$ of bounded degree that are invariant under permutation of the variables.

Define $e_0:=1$, for $n\ge 1$
\begin{align*}
 e_{n}:=\sum_{1\le i_1<i_2<\cdots<i_n}x_{i_1}x_{i_2}\cdots x_{i_n},
\end{align*}
 and for an integer partition $\lambda=(\lambda_{1}\ge \lambda_{2}\ge 
\cdots)$ (i.e., a weakly decreasing finite sequence of positive integers) 
$e_{\lambda}:=\prod_i e_{\lambda_i}$. We call 
$e_{\lambda}$ the \emph{elementary symmetric function} corresponding to the partition $\lambda$.

It is known that for $n \ge 0 $, the set $\{e_{\lambda}\,\mid\, \lambda \vdash n\}$ is a basis 
for the $n$-th homogeneous graded component of $\Lambda$, where the grading is with respect 
to degree. See \cite{Macdonald1995} and \cite{Stanley1999} for more information about symmetric 
functions.

Note that if we make the specialization $x_i \mapsto 0$ in $\Lambda$ for all 
$i\ge3$ then $e_1  \mapsto x_1+x_2$, $e_2 \mapsto x_1x_2$ and $e_i \mapsto 0$ for all $i\ge3$. 
Thus 
for a partition $\lambda$ of $n$ (i.e., $\sum_{i}\lambda_i=n$) the symmetric function $e_{\lambda} 
\mapsto 0$ unless 
$\lambda=(2^{j},1^{n-2j})$ for some $j\in \NN$. In that case,
\[
 e_{(2^{j},1^{n-2j})}\mapsto (x_1x_2)^j(x_1+x_2)^{n-2j}.
\]
If we further replace $x_1\mapsto 1$ and $x_2\mapsto t$ we obtain
\[
 e_{(2^{j},1^{n-2j})}\mapsto t^j(1+t)^{n-2j}.
\]
In other words, the elementary basis in two variables is equivalent to the $\gamma$ basis.
A consequence of this observation is that another possible approach to conclude the 
$\gamma$-positivity of a palindromic polynomial $f(t)$ is to find an  $e$-positive symmetric 
function $F(x_1,x_2,\dots)$ such that $f(t)=F(1,t,0,0,\dots)$.

\subsection{Colored combs and comb type of a normalized tree} A \emph{colored comb} is a  
normalized binary tree $T$ together with a function $\clr$ that assigns positive integers in $\PP$ 
to the internal nodes of $T$ and  that satisfies the following
coloring restriction:  for each internal node $x$ whose right child $R(x)$ is not a leaf, 
\begin{align}\label{equation:combcondition}
 \clr(x)>\clr(R(x)).
\end{align}

Note that the set of colored combs that only use the colors $1$ and $2$ are the same as 
the bicolored combs defined in Section \ref{section:binarytrees}. We denote $\mcomb_n$ the set of 
colored combs with $n$ leaves. Figure \ref{fig:combtype} shows an 
example of a colored comb. 

\begin{figure}[ht]
        \centering
        \usetikzlibrary{shapes,snakes}

\begin{tikzpicture}[thick,scale=0.8]

\begin{scope}[xshift=0cm,yshift=1cm]
 \draw [color=blue] (1.5,4)  node (blue){$1$};
\draw [circle,color=red] (1.5,3.5)  node (red){$2$};
 \draw [color=brown] (1.5,2.9)  node (brown){$3$};
\tikzstyle{every node}=[fill, draw,inner sep=4pt, minimum width=1pt,scale=0.8]
    
   \draw [color=blue] (1,4)  node (b){};
\draw [circle,color=red] (1,3.5)  node (r){};
    \draw [diamond,color=brown] (1,2.9)  node (g){};
\end{scope}

\node[fill=blue!20,blue!20,draw,rectangle,rounded corners,rotate=-45, minimum width=70pt,minimum 
height=30pt,scale=0.8] at  (4.5,2.4) {};
\node[fill=blue!20,blue!20,draw,rectangle,rounded corners,rotate=-22, minimum width=100pt,minimum 
height=30pt,scale=0.8] at  (7.25,3.5) {};
height=30pt,scale=0.8] at  (7.5,2) {};
\node[fill=blue!20,blue!20,draw,rectangle,rounded corners,rotate=45, minimum width=30pt,minimum 
height=30pt,scale=0.8] at  (2,1) {};
\node[fill=blue!20,blue!20,draw,rectangle,rounded corners,rotate=45, minimum width=30pt,minimum 
height=30pt,scale=0.8] at  (3,2) {};
\node[fill=blue!20,blue!20,draw,rectangle,rounded corners,rotate=45, minimum width=30pt,minimum 
height=30pt,scale=0.8] at  (6.5,1) {};
\node[fill=blue!20,blue!20,draw,rectangle,rounded corners,rotate=45, minimum width=30pt,minimum 
height=30pt,scale=0.8] at  (7.5,2) {};
\tikzstyle{every node}=[fill,draw,inner sep=0pt, minimum width=1 pt, scale=0.8]

    \draw [circle,color=red] (6,4)  node (i1){n};
    \draw [color=blue]  (8.5,3)  node (i2){N};
    \draw [circle,color=red]  (7.5,2)  node (i3){n};
    \draw [diamond,color=brown]  (6.5,1)  node (i4){n};
    \draw[diamond, color=brown]  (4,3)  node (i5){n};
    \draw [circle,color=red]  (5,2)  node (i6){n};
    \draw [diamond,color=brown] (3,2)  node (i7){n};
    \draw[color=blue]  (2,1)  node (i8){N};
\tikzstyle{every node}=[inner sep=1pt, minimum width=14pt,scale=0.8]

    \draw (4.3,1)  node (l1){2};
    \draw (5.5,0)  node (l2){3};
    \draw (1,0)  node (l3){1};
    \draw (3,0)  node (l4){4};
    \draw (6,1)  node (l5){5};
    \draw (3.7,1)  node (l6){6};
    \draw (7.5,0)  node (l7){7};
    \draw (9.5,2)  node (l8){9};
    \draw (8.5,1)  node (l9){8};

    \draw (i1) --  (i2) ;
    \draw (i1) --  (i5) ;
    \draw (i2) --  (i3) ;
    \draw (i2) --  (l8) ;
    \draw (i3) --  (i4) ;
    \draw (i3) --  (l9) ;
    \draw (i4) --  (l7) ;
    \draw (i4) --  (l2) ;
    \draw (i5) --  (i6) ;
    
    \draw (i5) --  (i7) ;
    \draw (i6) --  (l5) ;
    \draw (i6) --  (l1) ;
    \draw (i7) --  (l6) ;
    \draw (i7) --  (i8) ;
    \draw (i8) --  (l3) ;
    \draw (i8) --  (l4) ;
\end{tikzpicture}
 \caption{Example of a colored comb of comb type $(2,2,1,1,1,1)$}
\label{fig:combtype}
  \end{figure}
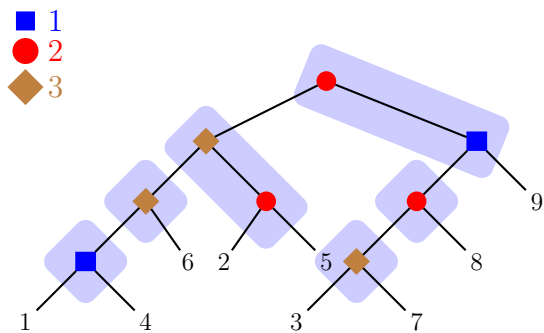

We can associate a type to  each $\Upsilon \in \nor_n$ in the following 
way: Let $\pi(\Upsilon)$ be the finest (set) partition of the set of internal nodes of 
$\Upsilon$ satisfying
\begin{itemize}
\item for every pair of internal nodes $x$ and $y$ such that $y$ is a right child of $x$, 
$x$ 
and $y$ belong to the same block of $\pi(\Upsilon)$.
\end{itemize}
We define the \emph{comb type} $\lambda(\Upsilon)$ of $\Upsilon$ to be the integer
partition whose parts are the sizes of the blocks of $\pi(\Upsilon)$.

Note that the coloring condition (\ref{equation:combcondition}) is closely related to the 
comb type of a normalized tree. The coloring 
condition 
implies that in a colored comb $\Upsilon$ there are no repeated colors in each block $B$ of 
the partition $\pi(\Upsilon)$ associated to $\Upsilon$.
So after choosing 
$|B|$ 
different colors for the 
internal nodes of $\Upsilon$ in $B$,  there is a unique way to assign the colors such 
that $\Upsilon$ is a colored comb (the colors must decrease towards the 
right in each block of $\pi(\Upsilon)$). 
In Figure \ref{fig:combtype} this relation is illustrated.

For a colored comb $C$ denote $\mu(C)$ the sequence of nonnegative integers such that 
$$\mu(C)(j):=| \{x \text{ a internal node in $C$ }\mid\, \clr(x)=j\}|.$$ Let 
$\xx^{\mu}:=x_1^{\mu(1)}x_2^{\mu(2)}\cdots$ and

$$F_{\mcomb_n}(\xx):= \sum_{C \in \mcomb_n}\xx^{\mu(C)}.$$

The following theorem is a consequence of the definition of a colored comb, the definition of 
the symmetric functions $e_i(\xx)$ and the observations 
above (see \cite{Dleon2014}).

\begin{theorem}[\cite{Dleon2014}]\label{theorem:epositivity} For $n \ge 1$
  $$F_{\mcomb_n}(\xx)=\sum_{T \in \nor_n}e_{\lambda(T)}(\xx).$$
\end{theorem}

Note that $F_{\mcomb_n}(1,t,0,0,\dots)=\sum_{C\in \comb_n}t^{\red{C}}=T_n(t)$ and so Theorem 
\ref{theorem:epositivity} is a generalization of Theorem 
\ref{theorem:gammacoefficientsrightdescents}.

\begin{remark}
 Versions of Theorem \ref{theorem:epositivity} can also be given in terms of a completely different 
type on the set $\nor_n$ corresponding to a family of multicolored Lyndon trees and also in terms 
of colored Stirling permutations, see \cite{Dleon2014}.
\end{remark}

\bibliographystyle{abbrv}
\bibliography{gammapolynomial}

\end{document}